\documentclass{aptpub}

\authornames{Thierry Klein, Agn\`es Lagnoux, Pierre Petit} 
\shorttitle{A conditional Berry-Esseen inequality} 

	\usepackage{lmodern}				

	\usepackage[T1]{fontenc}			

	\usepackage[utf8x]{inputenc}		

\usepackage{bbm}					
\usepackage{stmaryrd}	

	\usepackage{graphicx}				
	\usepackage{color}					

	\usepackage{version}					

	\usepackage{geometry}

\newcommand{\abs}[1]{\left\lvert #1 \right\rvert}
\newcommand{\norme}[1]{\left\lVert #1 \right\rVert}

\newcommand{\defeq}{\mathrel{\mathop:}=}
\newcommand{\eqdef}{\mathrel{=}:}

\newcommand{\ensnombre}[1]{\mathbb{#1}}%
\newcommand{\Z}{\ensnombre{Z}}

\newcommand{\R}{\ensnombre{R}}

\newcommand{\intervalle}[4]{\left#1#2\mathclose{}\mathpunct{},#3\right#4}%
\newcommand{\intervalleff}[2]{\intervalle{[}{#1}{#2}{]}}%
\newcommand{\intervalleoo}[2]{\intervalle{]}{#1}{#2}{[}}%
\newcommand{\intervallentff}[2]{\intervalle{\llbracket}{#1}{#2}{\rrbracket}}%
\newcommand{\indic}{\mathbbm{1}}
\newcommand{\Prob}{\mathbb{P}}
\newcommand{\Loi}{\mathcal{L}}
\newcommand{\Espe}{\mathbb{E}}

\DeclareMathOperator{\Var}{Var}
\DeclareMathOperator{\Cov}{Cov}

\iflanguage{french}%
{

\newcommand{\intervalle}[4]{\left#1#2\mathclose{}\mathpunct{},#3\right#4}%
\newcommand{\intervalleff}[2]{\intervalle{[}{#1}{#2}{]}}%
\newcommand{\intervalleoo}[2]{\intervalle{]}{#1}{#2}{[}}%
\newcommand{\intervallentff}[2]{\intervalle{\llbracket}{#1}{#2}{\rrbracket}}%
\newcommand{\intervalle}[4]{\left#1#2,#3\right#4}%
\newcommand{\intervalleff}[2]{\intervalle{[}{#1}{#2}{]}}%
\newcommand{\intervalleoo}[2]{\intervalle{(}{#1}{#2}{)}}%
\newcommand{\intervallentff}[2]{\intervalle{\llbracket}{#1}{#2}{\rrbracket}}%
}

	\usepackage{pstricks}
	\usepackage{amscd}
	\usepackage{ifmtarg}
	\usepackage{hyperref}

\makeatletter
\newcommand*{\rva}[3][]{%
	\@ifmtarg{#1}{
		#2_{#3}
	}{%
		#2_{#3,#1}
	}%
}
\makeatother

\newcommand*{\ssdd}[1]{\sigma_{#1}}
\newcommand*{\tm}[1]{\rho_{#1}}

\newcommand*{\bnxm}{\ssdd{\rva{X}{n}} N_n^{1/2}}
\newcommand*{\bnym}{\ssdd{\rva{Y}{n}} N_n^{1/2}}
\newcommand*{\bnx}{\ssdd{\rva{X}{n}}^{-1} N_n^{-1/2}}
\newcommand*{\bny}{\ssdd{\rva{Y}{n}}^{-1} N_n^{-1/2}}

\newcommand*{\ctll}{c_1}
\newcommand*{\cvarX}{c_2}
\newcommand*{\cvarXtilde}{\tilde{c}_2}
\newcommand*{\crhoX}{c_3}

\newcommand*{\cvarY}{c_4}
\newcommand*{\cvarYtilde}{\tilde{c}_4}
\newcommand*{\crhoY}{c_5}

\newcommand*{\ccorr}{c_6}
\newcommand*{\cfcX}{c_7}

\newcommand*{\cexpU}{d_1}
\newcommand*{\cvarU}{d_2}
\newcommand*{\cvarUprime}{d_2'}
\newcommand*{\cvarUseconde}{d_2''}
\newcommand*{\cvarUtierce}{d_2'''}
\newcommand*{\clemdeux}{C_4}
\newcommand*{\cn}{\gamma_n}

\renewcommand*{\th}{th}

\newcommand*{\leqp}{\leqslant}
\newcommand*{\geqp}{\geqslant}

\hypersetup{%
colorlinks=true,%
linkcolor = black,%
urlcolor = blue,%
citecolor = blue}

\makeatletter
\def\input@path{{./images/}}
\makeatother

\begin{document}

\title{A conditional Berry-Esseen inequality} 

\authorone[ENAC and Institut de
Math\'ematiques de Toulouse, UMR5219 -- Universit\'e de Toulouse, CNRS]{Thierry Klein} 
\addressone{ENAC, 7 avenue Edouard Belin, F-31400 Toulouse, France} 

\authortwo[Institut de
Math\'ematiques de Toulouse, UMR5219 -- Universit\'e de Toulouse, CNRS]{Agn\`es Lagnoux} 
\addresstwo{UT2, F-31058 Toulouse, France} 

\authorthree[Institut de
Math\'ematiques de Toulouse, UMR5219 -- Universit\'e de Toulouse, CNRS]{Pierre Petit} 
\addressthree{UPS IMT, F-31062 Toulouse Cedex 9, France } 

\begin{abstract}
\noindent As an extension of a central limit theorem established by Svante Janson, we prove a Berry-Esseen inequality for a sum of independent and identically distributed random variables conditioned by a sum of independent and identically distributed integer-valued random variables.
\end{abstract}



\section{Introduction}

As pointed out by Svante Janson in his seminal work \cite{Janson01}, in many random combinatorial problems, the interesting statistic is the sum of independent and identically distributed (i.i.d.)\ random variables conditioned on some exogenous integer-valued random variable. In general, the exogenous random variable is itself a sum of integer-valued random variables. Here, we are interested in the law of $N^{-1} (Y_1 + \dots + Y_N)$ conditioned on a specific value of $X_1 + \dots + X_N$ that is to say in the conditional distribution
\[
\Loi_N \defeq \Loi(N^{-1} (Y_1 + \dots + Y_N) \ |\ X_1 + \dots + X_N = m) ,
\]
where $m$ and $N$ are integers and the $(X_i, Y_i)$ for $1 \leqp i \leqp N$ are i.i.d.\ copies of a vector $(X, Y)$ of random variables with $X$ integer-valued.

In \cite{Janson01}, Janson proves a general central limit theorem (with convergence of all moments) for this kind of conditional distribution under some reasonable assumptions and gives several applications in classical combinatorial problems: occupancy in urns, hashing with linear probing, random forests, branching processes, etc. Following this work, one  natural question arises: is it possible to obtain a general Berry-Esseen inequality for these models?

The first Berry-Esseen inequality for a conditional model is given by Malcolm P.\ Quine and John Robinson in \cite{QR82}. They study the particular case of the occupancy problem, i.e., the case when the random variable $X$ is Poisson distributed and $Y = \indic_{\{ X = 0 \} }$. Up to our knowledge, it is the only result in that direction for this kind of conditional distribution.

Our paper is organized as follows. In Section \ref{sec:BE}, we present the model and we state our main results (Theorems \ref{th:BE_cond_strong} and \ref{th:BE_cond_strong_U}). In Section \ref{sec:example}, we describe classical examples. The last section is dedicated to the proofs.

\section{Conditional Berry-Esseen inequality}\label{sec:BE}

For all $n \geqslant 1$, we consider a vector of random variables $(\rva{X}{n},\rva{Y}{n})$ such that $\rva{X}{n}$ is integer-valued and $\rva{Y}{n}$ real-valued.
Let $N_n$ be a natural number such that $N_n \to \infty$ as $n$ goes to infinity.
Let $(\rva[i]{X}{n},\rva[i]{Y}{n})_{ 1\leqp i\leqp N_n}$ be an i.i.d.\ sample distributed as $(\rva{X}{n},\rva{Y}{n})$ and define 
\[
\rva[k]{S}{n} \defeq \sum_{i=1}^{k} \rva[i]{X}{n} \quad \text{and} \quad  \rva[k]{T}{n}\defeq\sum_{i=1}^{k} \rva[i]{Y}{n},
\]
for $k\in \intervallentff{1}{N_n}$. To lighten notation, define $S_n \defeq \rva[N_n]{S}{n}$ and $T_n \defeq \rva[N_n]{T}{n}$. Let $m_n \in \Z$ be such that $\Prob(S_n = m_n) > 0$. The purpose of the paper is to prove a Berry-Esseen inequality for the conditional distributions 
\[
\mathcal{L}(U_n) \defeq \mathcal{L}( T_n | S_n = m_n) .
\]

\begin{assumption}\label{hyp}

\renewcommand{\theenumi}{(A\arabic{enumi})}
\renewcommand{\labelenumi}{\theenumi}
Suppose that there exist positive constants $\ctll$, $\cvarXtilde$, $\cvarX$, $\crhoX$, $\cvarYtilde$, $\cvarY$, $\crhoY$, $\ccorr$, $\cfcX$, and $\eta_0$, such that:
\begin{enumerate}
\setcounter{enumi}{\theassumption}
\item \label{ass:tll}
$ \cn\defeq 2\pi \ssdd{\rva{X}{n}} N_n^{1/2}\Prob(S_n = m_n) \geqslant \ctll$;
\item \label{ass:var_X} $\cvarXtilde \leqslant \ssdd{\rva{X}{n}} \defeq \Var\left(\rva{X}{n}\right)^{1/2} \leqslant \cvarX$;
\item \label{ass:rho_X} $\tm{\rva{X}{n}} \defeq \Espe\bigl[ \abs{\rva{X}{n}-\Espe[\rva{X}{n}]}^3 \bigr] \leqslant \crhoX \ssdd{\rva{X}{n}}^3$;

\item \label{ass:var_Y} $\cvarYtilde \leqslant \ssdd{\rva{Y}{n}} \defeq \Var\left(\rva{Y}{n}\right)^{1/2} \leqslant \cvarY$;
\item \label{ass:rho_Y} $\tm{\rva{Y}{n}} \defeq \Espe\bigl[\abs{\rva{Y}{n}-\Espe[\rva{Y}{n}]}^3\bigr] \leqslant \crhoY \ssdd{\rva{Y}{n}}^3$;

\item \label{ass:corr} the correlations $r_n \defeq \Cov\left(\rva{X}{n},\rva{Y}{n}\right) \ssdd{\rva{X}{n}}^{-1} \ssdd{\rva{Y}{n}}^{-1}$ satisfy $|r_n| \leqslant \ccorr < 1$;
\item \label{ass:fc_XY} for $\rva{Y'}{n} \defeq \rva{Y}{n} -\Espe[\rva{Y}{n}] - \Cov(\rva{X}{n},\rva{Y}{n}) \ssdd{\rva{X}{n}}^{-2} (\rva{X}{n}-\Espe[\rva{X}{n}] )$, for all $s \in \intervalleff{-\pi}{\pi}$, and for all $t \in \intervalleff{-\eta_0}{\eta_0}$,
\[
\abs{\Espe\bigl[ e^{i(s\rva{X}{n} + t\rva{Y'}{n})} \bigr]} \leqslant 1 - \cfcX \bigl( \ssdd{\rva{X}{n}}^2 s^2 + \ssdd{\rva{Y'}{n}}^2 t^2 \bigr) .
\]
\end{enumerate}
\renewcommand{\theenumi}{\alph{enumi}}
\renewcommand{\labelenumi}{\theenumi.}
\end{assumption}

Obviously, Assumption \ref{hyp} is very close to the set of assumptions of the central limit theorem established in \cite[Theorem 2.3]{Janson01}. 
In particular, \ref{ass:tll} is a consequence of
$m_n = N_n\Espe[\rva{X}{n}]+O\bigl(\ssdd{\rva{X}{n}} N_n^{1/2}\bigr)$,
\ref{ass:rho_X}, and \ref{ass:fc_XY} (see the proof of Theorem 2.3 in \cite{Janson01}).
By \cite[Lemma 4.1.]{Janson01}, $\ssdd{\rva{X}{n}}^2 \leqslant 4 \Espe[\abs{X - \Espe[X]}^3]$ , so $\cvarXtilde$ can be chosen as $1/(4\crhoX)$. \ref{ass:corr} is not very restricting and holds in the examples provided in Section \ref{sec:example}.
Following \cite{Janson01}, we introduce $\rva{Y'}{n}$ in \ref{ass:fc_XY} in order to work with a centered variable uncorrelated with $\rva{X}{n}$.
If $(X, Y')$ is a vector of centered and uncorrelated random variables, then
\begin{align*}
\abs{\Espe\bigl[ e^{i(sX + tY')} \bigr]} & = 1 - \frac{1}{2} \bigl( \ssdd{X}^2 s^2 + \ssdd{Y'}^2 t^2 \bigr) + o(s^2 + t^2) ,
\end{align*}
so \ref{ass:fc_XY} is reasonable if the vectors $(\rva{X}{n}, \rva{Y'}{n})$ are identically distributed.


\begin{proposition}\label{cor:cv_distrib}
Assume that
\[
m_n = N_n\Espe[\rva{X}{n}]+O\bigl(\ssdd{\rva{X}{n}} N_n^{1/2}\bigr),
\]
that $(\rva{X}{n}, \rva{Y}{n})$ converges in distribution to $(X, Y)$ as $n \to \infty$, and that, for every fixed $r > 0$,
\[
\limsup_{n \to \infty} \Espe\left[|\rva{X}{n}|^r\right] < \infty \quad \text{and} \quad \limsup_{n \to \infty} \Espe\left[|\rva{Y}{n}|^r\right ]< \infty .
\]
Suppose further that the distribution of $X$ has span 1 and that $Y$ is not almost surely equal to an affine function $c+dX$ of $X$. Then, Assumption \ref{hyp} is satisfied.
\end{proposition}

The proof is omitted since the proposition relies on Corollary 2.1 and Theorem 2.3 in \cite{Janson01}.

%
%

\begin{theorem}\label{th:BE_cond_strong}
Under Assumption \ref{hyp}, $\tau_n^2 \defeq \ssdd{\rva{Y}{n}}^2 (1-r_n^2)  > 0$ and we have
\begin{equation}\label{eq:be}
\sup_{x \in \R} \left|\Prob\left(\frac{U_n-N_n \Espe\left[\rva{Y}{n}\right]-r_n \ssdd{\rva{Y}{n}}\ssdd{\rva{X}{n}}^{-1}(m_n - N_n\Espe\left[\rva{X}{n}\right])}{N_n^{1/2}\tau_n}\leqp x \right)-\Phi(x)\right|\leqp \frac{C}{N_n^{1/2}},
\end{equation}
where $\Phi$ denotes the standard normal cumulative distribution function and $C$ is a positive constant that only depends on $\cvarXtilde$, $\cvarX$, $\crhoX$, $\cvarYtilde$, $\cvarY$, $\crhoY$, $\ccorr$, $\cfcX$, $\eta_0$, and $\ctll$.
\end{theorem}

Remark that the standardization of the variables $U_n$ involved in \eqref{eq:be} is not the natural one. The following theorem fixes this default of standardization.

\begin{proposition}\label{prop:moment_U}
Under \ref{ass:tll}, \ref{ass:rho_X}, \ref{ass:var_Y}, \ref{ass:rho_Y}, and \ref{ass:fc_XY}, there exist two positive constants $\cexpU$ and $\cvarU$ depending only on $\crhoX$, $\cvarY$, $\crhoY$, $\cfcX$, and $\ctll$ such that, for $N_n \geqslant 3$,
\begin{equation}\label{eq:moment}
\abs{\Espe\left[U_n\right] - N_n \Espe[\rva{Y}{n}] - r_n \ssdd{\rva{Y}{n}} \ssdd{\rva{X}{n}}^{-1} (m_n - N_n\Espe[\rva{X}{n}])} \leqslant \cexpU
\end{equation}
and
\begin{equation}\label{eq:moment2}
\abs{\Var\left(U_n\right) - N_n \tau_n^2} \leqslant \cvarU N_n^{1/2}.
\end{equation}
\end{proposition}

\begin{theorem}\label{th:BE_cond_strong_U}
Under Assumption \ref{hyp}, we have
\begin{equation}\label{eq:be_2}
\sup_{x \in \R} \left|\Prob\left(\frac{U_n- \Espe\left[U_n\right]}{\Var\left(U_n\right)^{1/2}}\leqp x \right)-\Phi(x)\right|\leqp \frac{\widetilde{C}}{N_n^{1/2}},
\end{equation}
where $\widetilde{C}$ is a constant that only depends on $\cvarXtilde$, $\cvarX$, $\crhoX$, $\cvarYtilde$, $\cvarY$, $\crhoY$, $\ccorr$, $\cfcX$, $\eta_0$, and $\ctll$.
\end{theorem}

Furthermore, as in \cite{Janson01}, the results of Theorems  \ref{th:BE_cond_strong} and \ref{th:BE_cond_strong_U} simplify considerably in the special case when the vector  $(\rva{X}{n}, \rva{Y}{n})$ does not depend on $n$, that is to say when we consider an i.i.d.\ sequence instead of a triangular array. This is a consequence of Proposition \ref{cor:cv_distrib}.


\section{Classical examples}\label{sec:example}

In this section, we describe the examples mentioned in \cite{Janson01} and \cite{Holst79}. Each of them satisfies the assumptions of Proposition \ref{cor:cv_distrib}, as shown in \cite{Janson01}, leading to a Berry-Esseen inequality.

\subsection{Occupancy problem}\label{exocc}

In the classical occupancy problem, $m$ balls are thrown uniformly at random into $N$ urns. The resulting numbers of balls $(Z_1, \dots, Z_N)$ have a multinomial distribution. It is well known that $(Z_1, \dots, Z_N)$ is also distributed as $(X_1, \cdots, X_N)$ conditioned on $\{ \sum_{i=1}^N X_i = m \}$, where the random variables $X_i$ are i.i.d., with $X_i \sim \mathcal{P}(\lambda),$ for any arbitrary $\lambda > 0$. The classical occupancy problem studies the number of empty urns $U = \sum_{i=1}^N \indic_{\{Z_i = 0\}}$, which is distributed as $ \sum_{i=1}^N \indic_{\{X_i = 0\}}$ conditioned on $\{ \sum_{i=1}^N X_i = m \}$. Now, if $m = m_n \to \infty$ and $N = N_n \to \infty$ with $m_n/N_n \to \lambda \in \intervalleoo{0}{\infty}$, we can take $\rva{X}{n} \sim \mathcal{P}(\lambda_n)$ with $\lambda_n \defeq m_n/N_n$, $\rva{Y}{n} = \indic_{\{\rva{X}{n} = 0\}}$, and apply Proposition \ref{cor:cv_distrib} to obtain a Berry-Esseen inequality for $U_n = \sum_{i=1}^{N_n} \indic_{\{Z_i = 0\}}$.

\begin{remark}
In \cite{QR82}, the authors prove a Berry-Esseen inequality for the occupancy problem in a more general setting: the probability of landing in each urn may be different. The tools they developed will be used in the sequel to prove our results.  
\end{remark}

\begin{remark}
Here, we need a result for triangular arrays, and not only for i.i.d.\ sequences. Indeed, if we took $\rva{X}{n} = X$ with $X \sim \mathcal{P}(\lambda)$, we would only have
\[
m_n = N_n(\lambda + o(1)) = N_n \Espe[\rva{X}{n}] + o(N_n) .
\]
But Proposition \ref{cor:cv_distrib} requires 
\[
m_n = N_n \Espe[X] + O(N_n^{1/2}) ,
\]
which is stronger. This remark goes for the following examples too.
\end{remark}

\subsection{Bose-Einstein statistics}

This example is borrowed from \cite{Holst79} (see also \cite{Feller68}). Consider $N$ urns and put $m$ indistinguishable balls in the urns in such a way that each distinguishable outcome has the same probability $1/ \binom{m+N-1}{m}$. Let $Z_k$ be the number of balls in the $k$\th{} urn. It is well known that $(Z_1, \dots, Z_N)$ is distributed as $(X_1, \dots, X_N)$ conditioned on $\{\sum_{i=1}^N X_i = m\}$, where the random variables $X_i$ are i.i.d., with $X_i \sim \mathcal{G}(p),$ for any arbitrary $p \in \intervalleoo{0}{1}$. If $m = n$, $N = N_n \to \infty$ with $N_n/n \to p$, take $\rva{X}{n} \sim \mathcal{G}(p_n)$ with $p_n = N_n/n$ to obtain a Berry-Esseen inequality for any sequence of variables of the type $U_n = \sum_{i=1}^{N_n} f(Z_i)$.

\subsection{Branching processes}

Consider a Galton-Watson process, beginning with one individual, where the number of children of an individual is given by a random variable $X$ having finite moments. Assume further that $\Espe[X]=1$. We number the individuals as they appear. Let $X_i$ be the number of children of the $i$\th{} individual and $S_k \defeq \sum_{i=1}^k X_i$. It is well known (see \cite[Example 3.4]{Janson01} and the references therein) that the total progeny $S_N+1$ is $N \geqp 1$ if and only if
\begin{equation} \label{GW1}
\forall k \in \{ 0, \dots , N-1 \} \quad S_k \geqp k \quad \text{and} \quad S_N = N-1 .
\end{equation}
This type of conditioning is different from the one studied in the present paper, but by \cite[Corollary 2]{Wendel75} and \cite[Example 3.4]{Janson01}, if we ignore the cyclical order of $X_1, \dots, X_N$, it is proven that $X_1, \dots, X_N$ have the same distribution conditioned on \eqref{GW1} as conditioned on $\{ S_N = N-1 \}$. Applying Proposition \ref{cor:cv_distrib} with $N = n$ and $m = n-1$, we obtain a Berry-Esseen inequality for any sequence of variables $U_n$ distributed as $T_n = \sum_{i=1}^n f(X_i)$ conditioned on $\{ S_n = n-1 \}$. For instance, if $f(x) = \indic_{\{x = 3\}}$, $U_n$ is the number of individuals with three children given that the total progeny is $n$.

\subsection{Random forests} 

Consider a uniformly distributed random labeled rooted forest with $m$ vertices and $N$ roots with $N < m$. Without loss of generality, we may assume that the vertices are $1, \dots, m$ and, by symmetry, that the roots are the first $N$ vertices. Following \cite{Janson01}, this model can be realized as follows. The sizes of the $N$ trees in the forest are distributed as $(X_1, \dots , X_N)$ conditioned on $\{ \sum_{i=1}^N X_i = m \}$, where the random variables $X_i$ are i.i.d.\ and Borel distributed for any arbitrary parameter $\mu \in \intervalleoo{0}{1}$, i.e.
\[
\Prob(X_i = l) = e^{-\mu l} \frac{(\mu l)^{l-1}}{l!}
\]
(see, e.g., \cite{FPV98} or \cite{Janson01a} for more details). Then, the $i$\th{} tree is drawn uniformly among the trees of size $X_i$.  Proposition \ref{cor:cv_distrib} provides a Berry-Esseen inequality for any sequence of variables of the type $U_n = \sum_{i=1}^{N_n} f(Z_i)$ where $N_n \to \infty$ and $Z_1$, ..., $Z_{N_n}$ are the sizes of the trees in the forest. For instance, if $f(x) = \indic_{\{ x = K \}}$, $U_n$ is the number of trees of size $K$ in the forest (see, e.g., \cite{Kolchin84,Pavlov77,Pavlov96}).

\subsection{Hashing with linear probing}

Hashing with linear probing is a classical model in theoretical computer science that appeared in the 60's. It has been studied from a mathematical point of view firstly in \cite{Knuth74}. For more details on the model, we refer to \cite{FPV98, Janson01a, Marckert01-1, Chassaing02, ChF03, Janson05}. The model describes the following experiment. One throws $n$ balls sequentially into $m$ urns at random with $m > n$; the urns are arranged in a circle and numbered clockwise. A ball that lands in an occupied urn is moved to the next empty urn, always moving clockwise. The length of the move is called the displacement of the ball and we are interested in the sum of all displacements which is a random variable denoted $d_{m,n}$. After throwing all balls, there are $N \defeq m-n$ empty urns. These divide the occupied urns into blocks of consecutive urns. We consider that the empty urn following a block belongs to this block.
Following \cite{Knuth98a,FPV98}, Janson \cite{Janson01a} proves that the lengths of the blocks and the sums of displacements inside each block are distributed as $(X_1, Y_1)$, ..., $(X_N, Y_N)$ conditioned on $\{ \sum_{i=1}^N X_i = m \}$, where the random vectors $(X_i, Y_i)$ are i.i.d.\ copies of a vector $(X, Y)$ of random variables: $X$ being Borel distributed with any arbitrary parameter $\mu \in \intervalleoo{0}{1}$ and $Y$ given $\{ X = l \}$ being distributed as $d_{l,l-1}$.
In particular, $d_{m, n}$ is distributed as $\sum_{i=1}^N Y_i$ conditioned on $\{ \sum_{i=1}^N X_i = m \}$. If $m = m_n \to \infty$ and $N = N_n = m_n - n \to \infty$ with $n/m_n \to \mu \in \intervalleoo{0}{1}$, we take $\rva{X}{n}$ following Borel distribution with parameter $\mu_n \defeq n/m_n$ to get a Berry-Esseen inequality for $d_{m_n, n}$, by Proposition \ref{cor:cv_distrib}.


\section{Proofs}\label{sec:preuve}

Remind that $U_n$ is distributed as $T_n$ conditioned on $\{S_n=m_n\}$.
Following the procedure of \cite{Janson01}, we consider the projection
\begin{align*}
\rva{Y'}{n} = \rva{Y}{n} -\Espe[\rva{Y}{n}]-\Cov(\rva{X}{n},\rva{Y}{n}) \ssdd{\rva{X}{n}}^{-2} (\rva{X}{n}-\Espe[\rva{X}{n}]).
\end{align*}
Then $\Espe[\rva{Y'}{n}] = 0$ and $\Cov(\rva{X}{n},\rva{Y'}{n}) = \Espe[\rva{X}{n} \rva{Y'}{n}] = 0$. Besides, \ref{ass:fc_XY} and \ref{ass:corr} are verified by $\rva{Y'}{n}$. By \ref{ass:corr},
\[
\ssdd{\rva{Y'}{n}}^2 = \ssdd{\rva{Y}{n}}^2 (1 - r_n^2) \in [\cvarYtilde^2(1 - \ccorr^2), \cvarY^2],
\]
so \ref{ass:var_Y} is satisfied by $\rva{Y'}{n}$. Finally, by Minkowski inequality, \ref{ass:rho_X} and \ref{ass:rho_Y}, and the fact that $\abs{r_n} \leqslant 1$,
\begin{align*}
\norme{\rva{Y'}{n}}_3 & \leqp \norme{\rva{Y}{n}-\Espe[\rva{Y}{n}]}_3 + \abs{r_n} \ssdd{\rva{X}{n}}\ssdd{\rva{Y}{n}}\ssdd{\rva{X}{n}}^{-2} \norme{\rva{X}{n}-\Espe[\rva{X}{n}]}_3 \\
 & \leqp \tm{\rva{Y}{n}}^{1/3} + \ssdd{\rva{Y}{n}} \ssdd{\rva{X}{n}}^{-1} \tm{\rva{X}{n}}^{1/3} \\
  & \leqslant \ssdd{\rva{Y}{n}}(\crhoX^{1/3} + \crhoY^{1/3})\\
  & \leqslant \ssdd{\rva{Y'}{n}}(1-c_6^2)^{-1/2}(\crhoX^{1/3} + \crhoY^{1/3}).
  \end{align*}
Hence, $\rva{Y'}{n}$ satisfies \ref{ass:rho_Y}. Consequently, all conditions hold for the vector  $(\rva{X}{n},\rva{Y'}{n})$ too. Finally,
\[
T'_n \defeq \sum_{i=1}^{N_n} \rva[i]{Y'}{n} = T_n - N_n\Espe[\rva{Y}{n}] - \Cov(\rva{X}{n},\rva{Y}{n}) \ssdd{\rva{X}{n}}^{-2} (S_n-N_n \Espe[\rva{X}{n}]).
\]

So, conditioned on $\{S_n=m_n\}$, we have $T'_n=T_n - N_n \Espe[\rva{Y}{n}] - r_n \ssdd{\rva{Y}{n}} \ssdd{\rva{X}{n}}^{-1} (m_n - N_n\Espe[\rva{X}{n}])$. Hence the conclusions in Theorems \ref{th:BE_cond_strong} and \ref{th:BE_cond_strong_U} for $(\rva{X}{n},\rva{Y}{n})$ and $(\rva{X}{n},\rva{Y'}{n})$ are the same. Thus, it suffices to prove the theorems for $(\rva{X}{n},\rva{Y'}{n})$. In other words, we will henceforth assume that $\Espe\left[\rva{Y}{n}\right]=\Espe\left[\rva{X}{n} \rva{Y}{n}\right]=0$, $r_n=0$ and $\tau_n^2=\ssdd{\rva{Y}{n}}^2$. 
Moreover, the constants $\cvarY'$, $\cvarYtilde'$, $\crhoY'$, $\ccorr'$, and $\cfcX'$
for $(X,Y')$ are linked to that of $(X,Y)$ by the following relations: 
$\cvarY'=\cvarY$, $\cvarYtilde'=\cvarYtilde(1-\ccorr^2)^{1/2}$, $\crhoY'=(1-c_6^2)^{-3/2}(\crhoX^{1/3}+\crhoY^{1/3})^3$, $\ccorr'=0$, and $\cfcX'=\cfcX$. In the proofs, we omit the primes. 

\medskip

The proofs of Theorems \ref{th:BE_cond_strong} and \ref{th:BE_cond_strong_U} intensively rely on the use of Fourier transforms through the functions $\varphi_n$ and $\psi_n$ defined by
\begin{equation} \label{def:psi_n}
\varphi_n(s,t) \defeq \Espe\left[\exp\left\{is\left(\rva{X}{n}-\Espe\left[\rva{X}{n}\right]\right)+it\rva{Y}{n}\right\}\right]
\quad \text{and} \quad
\psi_n(t) \defeq2\pi \Prob(S_n=m_n) \Espe\left[\exp\left\{itU_n\right\}\right] .
\end{equation}
The controls of these functions (respectively the controls of their derivatives) needed in the proofs are postponed to Subection \ref{subsec:tech_res} in Lemmas \ref{lem:bartlett} and \ref{lem:maj_phi_n} (resp.\ in Lemma \ref{lem:maj_der_phi_n}). In particular, \eqref{eq:bartlett}, \eqref{eq:maj_expo_st}, \eqref{eq:maj_der_phi_n}, and \eqref{esp1b_array}
 will be used several times in the sequel.

\subsection{Proof of Theorem \ref{th:BE_cond_strong}}

We follow the classical proof of Berry-Esseen theorem (see e.g.\ \cite{Feller71}) combined with the procedure in \cite{QR82}. 
As shown in \cite{Loeve55} (page 285) or \cite{Feller71}, the left hand side of \eqref{eq:be} is dominated by 
\begin{equation*}
\frac{2}{\pi}\int_{0}^{\eta \bnym} \left|\frac{\psi_n(u\bny)}{2\pi\Prob(S_n=m_n)}-e^{-u^2/2}\right|\frac{du}{u}+\frac{24 \bny}{\eta \pi \sqrt{2\pi}},
\end{equation*}
where $\eta > 0$ is arbitrary. We choose to define
\begin{equation}\label{eq:eta}
\eta \defeq \min \bigg( \frac{2}{9} (\cvarY \crhoY)^{-1}, \eta_0 \bigg)>0.
\end{equation}
From \eqref{eq:bartlett} of Lemma \ref{lem:bartlett} and a Taylor's expansion,
\begin{align*}
&u^{-1}\left|\frac{\psi_n(u\bny)}{2\pi\Prob(S_n=m_n)}-e^{-u^2/2}\right| = u^{-1}e^{-u^2/2}\left|\frac{e^{u^2/2}\psi_n(u\bny)}{2\pi\Prob(S_n=m_n)}-1\right| \\
& \leqp  e^{-u^2/2}  \sup_{0\leqp \theta \leqp u} \left|\frac{\partial}{\partial t}\left[\frac{e^{t^2/2}\psi_n(t\bny)}{2\pi\Prob(S_n=m_n)}\right]_{t=\theta}\right| \\
& \leqp  \cn^{-1}e^{-u^2/2} \sup_{0\leqp \theta \leqp u} \left\{\int_{-\pi \bnxm}^{\pi \bnxm}   \left|\frac{\partial}{\partial t}\left[e^{t^2/2}\varphi_n^{N_n}\left(\frac{s}{\bnxm},\frac{t}{\bnym}\right)\right]_{t=\theta}\right| ds\right\}.
\end{align*}
By \ref{ass:tll}, $\cn\geqslant \ctll$. Now we split the integration domain of $s$ into 
\[
A_1 \defeq \left\{s:\; |s|< \varepsilon \bnxm\right\} \quad \textrm{and} \quad A_2 \defeq \left\{s:\; \varepsilon \bnxm\leqp |s| \leqp \pi \bnxm\right\},
\]
where 
\begin{equation}\label{eq:epsilon}
\varepsilon \defeq \min \bigg( \frac{2}{9} (\cvarX \crhoX)^{-1}, \pi \bigg)
\end{equation}
and decompose 
\begin{equation*}
u^{-1}\left|\frac{\psi_n(u\bny)}{2\pi\Prob(S_n=m_n)}-e^{-u^2/2}\right|\leqp \sup_{0 \leqp \theta\leqp u} \left[I_1(n,u, \theta) + I_2(n,u, \theta)\right],
\end{equation*}
where
\begin{align}
I_1(n, u, \theta)&= \cn^{-1}  \int_{A_1}  e^{-(u^2+s^2)/2} \abs{\frac{\partial}{\partial t} \left[e^{(t^2+s^2)/2}\varphi_n^{N_n}\left(\frac{s}{\bnxm},\frac{t}{\bnym}\right)\right]_{t=\theta}} ds, \label{def:I_1}\\
I_2(n, u, \theta)&= \cn^{-1}  e^{-u^2/2} \int_{A_2} \abs{\frac{\partial}{\partial t}\left[e^{t^2/2}\varphi_n^{N_n}\left(\frac{s}{\bnxm},\frac{t}{\bnym}\right)\right]_{t=\theta}} ds. \label{def:I_2}
\end{align}
Lemmas \ref{lem:I1} and \ref{lem:I2} state that there exists positive constants $C_1$ and $C_2$, only depending on $\cvarXtilde$, $\cvarX$, $\crhoX$, $\crhoY$, $\cfcX$, and $\ctll$, such that, for $N_n \geqslant \max(12^3\crhoX^2, 12^3\crhoY^2, 2)$, 
\begin{equation}\label{eq:I1}
\int_0^{\eta \bnym} \sup_{0\leqp \theta \leqp u} I_1(n, u, \theta) du \leqp \frac{C_1}{N_n^{1/2}},
\end{equation}
and
\begin{equation}\label{eq:I2}
\int_0^{\eta \bnym} \sup_{0\leqp \theta \leqp u} I_2(n, u, \theta)du \leqp \frac{C_2}{N_n^{1/2}}.
\end{equation}
So,
\[
\sup_{x \in \R} \left|\Prob\left(\frac{U_n}{N_n^{1/2}\ssdd{\rva{Y}{n}}}\leqp x \right)-\Phi(x)\right| \leqp \frac{C}{N_n^{1/2}}
\]
with
\begin{equation*} 
C \defeq \max \biggl( C_1 + C_2 +\frac{24}{\widetilde \cvarY \pi \sqrt{2\pi}} \biggl( \min \biggl( \frac{2}{9} \cvarY \crhoY, \eta_0 \biggr)\biggr)^{-1} , 12^{3/2}\crhoX, 12^{3/2}\crhoY, \sqrt{2} \biggr) .
\end{equation*}

\subsection{Proof of Proposition \ref{prop:moment_U}}

\paragraph{Proof of \eqref{eq:moment}}

 We adapt the proof given in \cite{Janson01}. Using the definition \eqref{def:psi_n} of $\Psi_n$, and differentiating under the integral sign of \eqref{eq:bartlett} of Lemma \ref{lem:bartlett}, we naturally have
\begin{align*}
 &\abs{\Espe\left[U_n\right]} = \abs{\frac{-i\psi_n'(0)}{2\pi \Prob(S_n=m_n)}} \nonumber\\
 & \leqslant  \cn^{-1}N_n \int_{-\pi \bnxm}^{\pi\bnxm} \abs{\frac{\partial \varphi_n}{\partial t}\left(\frac{s}{\bnxm},0\right)} \cdot \abs{\varphi_n^{N_n-1}\left(\frac{s}{\bnxm},0\right)} ds.
\end{align*}
Using \eqref{esp1b_array} of Lemma \ref{lem:maj_der_phi_n} with $t = 0$, \ref{ass:var_X}, \ref{ass:rho_X}, and \ref{ass:rho_Y}, we deduce
\[
\abs{\frac{\partial \varphi_n}{\partial t}\left(\frac{s}{\bnxm},0\right)} \leqp \frac{s^2}{2}\frac{\tm{\rva{Y}{n}}^{1/3}\tm{\rva{X}{n}}^{2/3}}{\ssdd{\rva{X}{n}}^2 N_n} \leqp \frac{\crhoX^{2/3} \cvarY \crhoY^{1/3}}{2 N_n} s^2.
\]
Then using \eqref{eq:maj_expo_st} of Lemma \ref{lem:maj_phi_n} (with $l=1$ and $t = 0$) and for $N_n \geqslant 3$,
\[
\int_{-\pi \bnxm}^{\pi\bnxm} \abs{\frac{\partial \varphi_n}{\partial t}\left(\frac{s}{\bnxm},0\right)} \cdot \abs{\varphi_n^{N_n-1}\left(\frac{s}{\bnxm},0\right)} ds
\leqp \frac{\crhoX^{2/3} \cvarY \crhoY^{1/3}}{2 N_n} \int_{-\infty}^{+\infty} s^2 e^{-2 \cfcX s^2/3} ds.
\]
So, \eqref{eq:moment} holds with
\[
\cexpU \defeq 2^{-1} \crhoX^{2/3} \cvarY \crhoY^{1/3} \ctll^{-1} \int_{-\infty}^{+\infty} s^2 e^{-2 \cfcX s^2/3} ds .
\]
 
\paragraph{Proof of \eqref{eq:moment2}}

Since $\tau_n^2 = \ssdd{\rva{Y}{n}}^2$ and $\Espe\left[U_n\right]$ is bounded, it suffices to show that the quantity $\abs{\Espe\left[U_n^2\right] - N_n \ssdd{\rva{Y}{n}}^2}$ is bounded by some $\cvarUprime N_n^{1/2}$ to prove \eqref{eq:moment2}. Proceeding as done previously,
\begin{align}
 & \Espe\left[U_n^2\right] = \frac{-\psi_n''(0)}{2\pi \Prob(S_n=m_n)}\nonumber\\
 & = - \cn^{-1} N_n(N_n-1)  \int_{-\pi \bnxm}^{\pi\bnxm}
 e^{-isv_n} \left(\frac{\partial \varphi_n}{\partial t}\left(\frac{s}{\bnxm},0\right)\right)^2\varphi_n^{N_n-2}\left(\frac{s}{\bnxm},0\right)ds\label{eq:moment3}\\
 & \quad - \cn^{-1} N_n  \int_{-\pi \bnxm}^{\pi\bnxm}  e^{-isv_n} \frac{\partial^2 \varphi_n}{\partial t^2}\left(\frac{s}{\bnxm},0\right)\varphi_n^{N_n-1}\left(\frac{s}{\bnxm},0\right)ds\label{eq:moment4}
\end{align}
where 
\begin{equation}\label{def:vn}
v_n \defeq (m_n - {N_n}\Espe\left[\rva{X}{n}\right])/(\ssdd{\rva{X}{n}} N_n^{1/2}).
\end{equation}
First, by \eqref{esp1b_array} of Lemma \ref{lem:maj_der_phi_n} with $t = 0$ and by \eqref{eq:maj_expo_st} of Lemma \ref{lem:maj_phi_n} (with $l=2$ and $t = 0$), one has, for $N_n \geqslant 3$,
\begin{align*}
 & \int_{-\pi \bnxm}^{\pi\bnxm} \abs{\frac{\partial \varphi_n}{\partial t}\left(\frac{s}{\bnxm},0\right)}^2 \abs{\varphi_n^{N_n-2}\left(\frac{s}{\bnxm},0\right)} ds \\
 & \qquad \leqp \frac{\crhoX^{4/3} \cvarY^2 \crhoY^{2/3}}{4 N_n^2} \int_{-\infty}^{+\infty} s^4 e^{- \cfcX s^2/3}ds.
\end{align*}
Finally, by \ref{ass:tll}, the term \eqref{eq:moment3} is bounded by
\begin{equation*}
\cvarUseconde \defeq \frac{\crhoX^{4/3} \cvarY^2 \crhoY^{2/3}}{4 \ctll} \int_{-\infty}^{+\infty} s^4 e^{- \cfcX s^2/3}ds.
\end{equation*}

Second, we study the term \eqref{eq:moment4}. We want to show that 
\[
\Delta_n: =  \cn^{-1} \int_{-\pi \bnxm}^{\pi\bnxm} e^{-isv_n} \frac{\partial^2 \varphi_n}{\partial t^2}\left(\frac{s}{\bnxm}, 0\right) \varphi_n^{N_n-1}\left(\frac{s}{\bnxm}, 0\right)ds + \ssdd{\rva{Y}{n}}^2
\]
is bounded by some $\cvarUtierce/N_n^{1/2}$. By \eqref{eq:bartlett} with $t=0$,
\[
\int_{-\pi \bnxm}^{\pi\bnxm} e^{-isv_n} \varphi_n^{N_n}\left(\frac{s}{\bnxm},0\right) ds = 2 \pi \Prob(S_n = m_n) \ssdd{\rva{X}{n}} N_n^{1/2} = \cn,
\]
so
\begin{align*}
\Delta_n & = \cn^{-1} \int_{-\pi \bnxm}^{\pi\bnxm} e^{-isv_n}  \left( \frac{\partial^2 \varphi_n}{\partial t^2}\left(\frac{s}{\bnxm},0\right) + \ssdd{\rva{Y}{n}}^2 \varphi_n\left(\frac{s}{\bnxm},0\right)\right) 
 \varphi_n^{N_n-1} \left(\frac{s}{\bnxm},0\right)ds \\
 &=  \cn^{-1} \int_{-\pi \bnxm}^{\pi\bnxm} e^{-isv_n} \Espe\bigg[ {\rva{Y}{n}}^2 f(s) \bigg] \cdot\varphi_n^{N_n-1} \left(\frac{s}{\bnxm},0\right)ds
\end{align*}
where \[
f(s) = -  \left(e^{is\bnx(\rva{X}{n}-\Espe[\rva{X}{n}])} -\Espe\Big[ e^{is\bnx(\rva{X}{n}-\Espe[\rva{X}{n}])} \Big]\right).
\]
Applying Taylor's theorem yields
\begin{align*}
\abs{f(s)} & \leqslant \abs{s} \sup_u \left\lvert - i\frac{\rva{X}{n}-\Espe[\rva{X}{n}]}{\bnxm} e^{iu\bnx(\rva{X}{n}-\Espe[\rva{X}{n}])} \right. 
\left. 
+ \Espe\bigg[ i\frac{\rva{X}{n}-\Espe[\rva{X}{n}]}{\bnxm} e^{iu\bnx(\rva{X}{n}-\Espe[\rva{X}{n}])} \bigg]\right\rvert \\
 & \leqslant \frac{\abs{s}}{N_n^{1/2}} \bigg( \abs{ \frac{\rva{X}{n}-\Espe[\rva{X}{n}]}{\ssdd{\rva{X}{n}}}} + \Espe\bigg[ \abs{ \frac{\rva{X}{n}-\Espe[\rva{X}{n}]}{\ssdd{\rva{X}{n}}}} \bigg] \bigg).
\end{align*}
Thus, using H\"older's inequality,
\begin{align*}
\abs{\Espe[{\rva{Y}{n}}^2 f(s)]} & \leqslant \frac{\abs{s}}{N_n^{1/2}} \Espe\bigg[ {\rva{Y}{n}}^2 \bigg( \abs{\frac{\rva{X}{n}-\Espe[\rva{X}{n}]}{\ssdd{\rva{X}{n}}}} + \Espe\bigg[ \abs{\frac{\rva{X}{n}-\Espe[\rva{X}{n}]}{\ssdd{\rva{X}{n}}}} \bigg] \bigg) \bigg] \\
 & \leqslant \frac{\ssdd{\rva{Y}{n}}^2 \abs{s}}{N_n^{1/2}} \bigg( \frac{\tm{\rva{Y}{n}}^{2/3}}{\ssdd{\rva{Y}{n}}^2} \frac{\tm{\rva{X}{n}}^{1/3}}{\ssdd{\rva{X}{n}}} + 1 \bigg)
\\
& \leqslant \frac{\abs{s} \cvarY^2}{N_n^{1/2}} \bigg(\crhoY^{2/3}\crhoX^{1/3}+1 \bigg) 
\end{align*}
where the last inequality is obtained using \ref{ass:var_X}, \ref{ass:rho_X}, \ref{ass:var_Y}, and \ref{ass:rho_Y}. Now, 
by \ref{ass:tll} and the upper bound in \eqref{eq:maj_expo_st} (with $l=1$ and $t = 0$), we get, for $N_n\geqslant 3$,
\[
\abs{\Delta_n} \leqslant \frac{\cvarY^2}{\ctll N_n^{1/2}} \bigg(\crhoY^{2/3}\crhoX^{1/3}+1 \bigg)  \int_{-\infty}^{+\infty} \abs{s} e^{-s^2 \cfcX (N_n-1)/N_n} ds \leqslant \frac{\cvarUtierce}{N_n^{1/2}},
\]
with
\begin{equation*}
\cvarUtierce \defeq \cvarY^2 \ctll^{-1}( \crhoY^{2/3} \crhoX^{1/3} +1) \int_{-\infty}^{+\infty} \abs{s} e^{-2 s^2 \cfcX/3} ds.
\end{equation*}
Finally, 
\[
\abs{\Var(U_n) - N_n \ssdd{\rva{Y}{n}}^2} \leqslant (\cexpU^2 + \cvarUseconde + \cvarUtierce) N_n^{1/2} \eqdef \cvarU N_n^{1/2} .
\]

Then the proof of \eqref{eq:moment2} is complete.

\subsection{Proof of Theorem \ref{th:BE_cond_strong_U}}

Write 
\begin{align*}
\abs{\Prob\left( \frac{U_n - \Espe[U_n]}{\Var\left(U_n\right)^{1/2}} \leqslant x \right) - \Phi(x)}
 & \leqslant \abs{\Prob\left( \frac{U_n}{N_n^{1/2} \ssdd{\rva{Y}{n}}} \leqp a_n x + b_n \right) - \Phi(a_n x + b_n)} \\
 & \hspace{5cm} + \abs{\Phi(a_n x + b_n) - \Phi(x)} \\
\end{align*}
where
\[
a_n \defeq \frac{\Var(U_n)^{1/2}}{N_n^{1/2} \ssdd{\rva{Y}{n}}} \quad \text{and} \quad b_n \defeq \frac{\Espe[U_n]}{N_n^{1/2} \ssdd{\rva{Y}{n}}}.
\]
The previous estimates of $\Espe[U_n]$ and $\Var(U_n)$ yield, 
\[
\abs{a_n - 1} \leqslant \abs{a_n^2 - 1} \leqslant \cvarU \cvarYtilde^{-2} N_n^{-1/2} \quad \text{and} \quad \abs{b_n} \leqslant \cexpU \cvarYtilde^{-1} N_n^{-1/2}.
\]
Then for $N_n^{1/2} \geqslant 2 \cvarYtilde^{-2} \cvarU$, $a_n \geqslant 1/2$ and applying Taylor's theorem to $\Phi$, one gets
\begin{align*}
\abs{\Phi(a_n x + b_n) - \Phi(x)} & \leqslant \abs{(a_n - 1) x + b_n} \sup_t \frac{e^{-t^2/2}}{\sqrt{2\pi}} \\
 & \leqslant \frac{N_n^{-1/2}}{\sqrt{2\pi}} \max(\cvarU \cvarYtilde^{-2}, \cexpU \cvarYtilde^{-1}) (\abs{x} + 1) e^{-(|x|/2 - \cexpU \cvarYtilde^{-1})^2/2}\\
\end{align*}
the supremum being over $t$ between $x$ and $a_n x + b_n$.  The last function in $x$ being bounded, we can define 
\[
C'\defeq \frac{1}{\sqrt{2\pi}}  \max(\cvarU \cvarYtilde^{-2}, \cexpU \cvarYtilde^{-1} ) \sup_{x \in \R} \Big[(\abs{x} + 1) e^{-(|x|/2 - \cexpU \cvarYtilde^{-1})^2/2} \Big] 
.
\] 
Finally, we apply \eqref{eq:be} and \eqref{eq:be_2} holds with $\widetilde{C}\defeq C+\max(C', 2\cvarYtilde^{-2} \cvarU)$.

\subsection{Technical results}\label{subsec:tech_res}

Recall that $v_n = (m_n - {N_n}\Espe\left[\rva{X}{n}\right])/(\ssdd{\rva{X}{n}} N_n^{1/2})$ and $\cn = 2\pi\Prob(S_n=m_n) \bnxm$. Moreover,
\begin{equation*}
\varphi_n(s,t) = \Espe\left[\exp\left\{is\left(\rva{X}{n}-\Espe\left[\rva{X}{n}\right]\right)+it \rva{Y}{n}\right\}\right]
\quad \text{and} \quad
\psi_n(t) = 2\pi \Prob(S_n=m_n) \Espe\left[\exp\left\{itU_n\right\}\right] .
\end{equation*}

\begin{lemma}\label{lem:bartlett} One has
\begin{equation}\label{eq:bartlett}
\psi_n(t)=\frac{1}{\ssdd{\rva{X}{n}} N_n^{1/2}} \int_{-\pi \ssdd{\rva{X}{n}} N_n^{1/2}}^{\pi \ssdd{\rva{X}{n}} N_n^{1/2}} e^{-isv_n} \varphi_n^{N_n}\left(\frac{s}{\bnxm},t\right) ds
\end{equation}
\end{lemma}

\begin{proof}
Indeed, since
\[
\int_{-\pi}^\pi e^{is(S_n - m_n)} ds = 2 \pi \indic_{\{ S_n = m_n \}},
\]
we have
\begin{align*}
\psi_n(t) & = 2\pi \Prob(S_n=m_n) \Espe\left[\exp\left\{itU_n\right\}\right] \\
 & = 2\pi \Espe\left[ \exp\left\{it T_n \right\} \indic_{\{S_n = m_n\}} \right] \\
 & = \int_{-\pi}^{\pi} \Espe\left[\exp\left\{is \left(S_n -m_n\right)+it  T_n \right\} \right]ds \\
 & = \int_{-\pi}^{\pi} e^{-is\left(m_n-N_n\Espe\left[\rva{X}{n}\right]\right)}\varphi_n^{N_n}(s,t)ds,
\end{align*}
which leads to \eqref{eq:bartlett}
after the change of variable $s'=s\ssdd{\rva{X}{n}} N_n^{1/2}$.
\end{proof}

Now we give controls on the function $\varphi_n$ and its partial derivatives (see Lemmas \ref{lem:maj_phi_n} and \ref{lem:maj_der_phi_n}).

\begin{lemma} \label{lem:maj_phi_n}
Under \ref{ass:fc_XY}, for any integer $l \geqslant 0$, $\abs{s} \leqslant \pi \ssdd{\rva{X}{n}} N_n^{1/2}$, and $\abs{t}  \leqslant \eta_0 \ssdd{\rva{Y}{n}} N_n^{1/2}$, one gets
\begin{equation} \label{eq:maj_expo_st}
\abs{\varphi_n^{N_n-l} \bigg(\frac{s}{\bnxm}, \frac{t}{\bnym} \bigg)} \leqslant e^{-(s^2+t^2) \cdot \cfcX \cdot (N_n - l)/N_n}.
\end{equation}
\end{lemma}

\begin{proof}
The proof is a mere consequence of the inequality $1 + x \leqslant e^x$ that holds for any $x\in\R$.
\end{proof}

\begin{lemma} \label{lem:maj_der_phi_n}
For any $s$ and $t$, one has
\begin{equation} \label{eq:maj_der_phi_n}
\left|\frac{\partial \varphi_n}{\partial t} \left(\frac{s}{\bnxm},\frac{t}{\bnym}\right)\right| \leqp \frac{\ssdd{\rva{Y}{n}}}{N_n^{1/2}} (|s|+|t|) ;
\end{equation}
and 
\begin{align}\label{esp1b_array}
\left\lvert \frac{\partial \varphi_n}{\partial t} \left(\frac{s}{\ssdd{\rva{X}{n}} N_n^{1/2}}, \frac{t}{\ssdd{\rva{Y}{n}} N_n^{1/2}} \right)\right\rvert & \leqslant \frac{\ssdd{\rva{Y}{n}}}{N_n^{1/2}}  |t| + \frac{\ssdd{\rva{Y}{n}}}{N_n} \bigg[ \frac{s^2}{2} \bigg( \frac{\tm{\rva{X}{n}}}{\ssdd{\rva{X}{n}}^3}\bigg)^{2/3} \bigg( \frac{\tm{\rva{Y}{n}}}{\ssdd{\rva{Y}{n}}^3}\bigg)^{1/3} \nonumber \\
 &  + \abs{st} \bigg( \frac{\tm{\rva{X}{n}}}{\ssdd{\rva{X}{n}}^3}\bigg)^{1/3} \bigg( \frac{\tm{\rva{Y}{n}}}{\ssdd{\rva{Y}{n}}^3}\bigg)^{2/3} + \frac{t^2}{2} \bigg( \frac{\tm{\rva{Y}{n}}}{\ssdd{\rva{Y}{n}}^3}\bigg) \bigg].
\end{align}
\end{lemma}

\begin{proof}
We apply Taylor's theorem to the function defined by
\[
(s,t) \mapsto f(s,t)=\frac{\partial \varphi_n}{\partial t} \left(\frac{s}{\bnxm},\frac{t}{\bnym}\right).
\]
We conclude to \eqref{eq:maj_der_phi_n} using
\[
\left|f(s,t)-f(0,0)\right|\leqp |s|\sup_{\theta, \theta' \in \intervalleff{0}{1}} \left|\frac{\partial f}{\partial s}\left(\theta s,\theta' t\right)\right|+|t|\sup_{\theta, \theta' \in \intervalleff{0}{1}} \left|\frac{\partial f}{\partial t}\left(\theta s,\theta' t\right)\right|
\]
and to \eqref{esp1b_array} using
\begin{align*}
\left|f(s,t)-f(0,0)\right|
 & \leqp |s| \left|\frac{\partial f}{\partial s}\left(0,0\right)\right|+|t| \left|\frac{\partial f}{\partial t}\left(0,0\right)\right| + \frac{s^2}{2}\sup_{\theta, \theta' \in \intervalleff{0}{1}} \left|\frac{\partial^2 f}{\partial^2 s}\left(\theta s,\theta' t\right)\right|\\
 & \qquad + |st|\sup_{\theta, \theta' \in \intervalleff{0}{1}} \left|\frac{\partial^2 f}{\partial t\partial s}\left(\theta s,\theta' t\right)\right|+\frac{t^2}{2}\sup_{\theta, \theta' \in \intervalleff{0}{1}} \left|\frac{\partial^2 f}{\partial^2 t}\left(\theta s,\theta' t\right)\right|.
\end{align*}
The partial derivatives of $f$ are estimated by mixed moments of $X_n$ and $Y_n$ and then bounded above by Hölder's inequality. 
\end{proof}

The following lemma is a result due to Quine and Robinson (\cite[Lemma 2]{QR82}).

\begin{lemma}\label{lem:lem2}
Define
\[
l_{1,n} \defeq \tm{\rva{X}{n}} \ssdd{\rva{X}{n}}^{-3} N_n^{-1/2} \qquad \text{and} \qquad l_{2,n} \defeq \tm{\rva{Y}{n}} \ssdd{\rva{Y}{n}}^{-3} N_n^{-1/2}.
\]
If $l_{1,n} \leqslant 12^{-3/2}$ and $l_{2,n} \leqslant 12^{-3/2}$, then, for all
\[
(s, t) \in R \defeq \left\{(s,t):\; |s|<\frac{2}{9}l_{1,n}^{-1}, |t|<\frac{2}{9}l_{2,n}^{-1}\right\},
\]
we have
\begin{align*}
\left\lvert \frac{\partial}{\partial t}\bigg[e^{(s^2+t^2)/2} \right.
 & \left. \varphi_n^{N_n}\left(\frac{s}{\bnxm},\frac{t}{\bnym}\right)\bigg]\right\rvert 
 \leqp \clemdeux(|s|+|t|+1)^3(l_{1,n}+l_{2,n})\exp\left\{\frac{11}{24}\left(s^2+t^2\right)\right\},
\end{align*}
with $\clemdeux \defeq 161$.
\end{lemma}

\begin{remark}
We make explicit the constant $\clemdeux$ appearing at the end of the proof of Lemma 2 in \cite{QR82}. For all $v$ and $s$ in $R_2$ as defined in \cite{QR82},  one has
\begin{align*}
\frac{(\abs{v} + 2\abs{s})}{(\abs{v} + \abs{s} + 1)^3(\ell_{1,n}+\ell_{2,n})} e^{-(v^2 + s^2)/24} 
 &  \leqslant 108 \cdot \sqrt{6} \cdot e^{-1/2} \leqslant 161.
\end{align*}

\end{remark}

By \ref{ass:var_X} and \ref{ass:rho_X},
\begin{equation*} 
l_{1,n} \leqslant \crhoX N_n^{-1/2} \leqslant  \cvarX \crhoX \ssdd{\rva{X}{n}}^{-1} N_n^{-1/2},
\end{equation*}
which implies that $\bnxm \leqp  \cvarX \crhoX l_{1,n}^{-1}$.
Similarly,
\begin{equation*} 
l_{2,n} \leqslant \crhoY N_n^{-1/2} \leqslant \cvarY \crhoY  \ssdd{\rva{Y}{n}}^{-1} N_n^{-1/2},
\end{equation*}
and $\bnym \leqp \cvarY \crhoY  l_{2,n}^{-1}$. Now we are able to establish \eqref{eq:I1}.

\begin{lemma}\label{lem:I1}
There exists a positive constant $C_1$, only depending on $\crhoX$, $\crhoY$, $\ctll$ such that, for $N_n \geqslant 12^3\max(\crhoX^2, \crhoY^2)$,
\begin{equation*}
\int_0^{\eta \bnym} \sup_{0\leqp \theta \leqp u} I_1(n, u, \theta) du \leqp \frac{C_1}{N_n^{1/2}}.
\end{equation*}
\end{lemma}

\begin{proof}
The definitions of $\eta$ in \eqref{eq:eta} and $\varepsilon$ in  \eqref{eq:epsilon} imply that, for  $s\in A_1$ and $u$ and $\theta$ as in the integral in the statement above, one has  
\begin{align*}
|s| & < \varepsilon \bnxm \leqslant \frac{2}{9} l_{1,n}^{-1} \quad
\text{and} \quad |\theta| \leqslant |u| \leqslant \eta \bnym \leqslant \frac{2}{9} l_{2,n}^{-1},
\end{align*}
which ensures that $(s, \theta) \in R$ as specified in Lemma \ref{lem:lem2}. Moreover, for $N_n \geqslant 12^3\max(\crhoX^2, \crhoY^2)$, $l_{1,n} \leqslant 12^{-3/2}$ and $l_{2,n} \leqslant 12^{-3/2}$. Now using Lemma \ref{lem:lem2} in \eqref{def:I_1} and by \ref{ass:tll}, we get
\begin{align*}
\int_0^{\eta \bnym}
 & \sup_{0 \leqp \theta \leqp u} I_1(n, u, \theta) du \\
 & \leqp \cn^{-1} \clemdeux (l_{1,n}+l_{2,n}) \int_0^{\eta \bnym}  \int_{A_1}  (|s|+|u|+1)^3  e^{-(s^2+u^2)/24} ds du \\
 & \leqp N_n^{-1/2} \ctll^{-1} \clemdeux (\crhoX + \crhoY) \int_{\R^2} (|s|+|u|+1)^3  e^{-(s^2+u^2)/24}dsdu
\end{align*}  
and the result follows with 
\[
C_1 \defeq \ctll^{-1} \clemdeux (\crhoX + \crhoY) \int_{\R^2} (|s|+|u|+1)^3  e^{-(s^2+u^2)/24} ds du .
\]
\end{proof}

\begin{remark}
Actually, Lemma \ref{lem:I1} is valid as soon as $N_n \geqslant \max(\crhoX^2, \crhoY^2)$: the constants in the proof of Lemma 2 in \cite{QR82} can be improved.
\end{remark}

Now we are able to prove \eqref{eq:I2}.

\begin{lemma}\label{lem:I2}
There exists a positive constant $C_2$, only depending on $\ctll$, $\cvarXtilde$, $\cvarX$, $\crhoX$, and $\cfcX$ such that, for $N_n \geqslant 2$,
\begin{equation*}
\int_0^{\eta \bnym} \sup_{0\leqp \theta \leqp u} I_2(n, u, \theta)du \leqp \frac{C_2}{ N_n^{1/2}}.
\end{equation*}
\end{lemma}

\begin{proof} We use the controls  \eqref{eq:maj_expo_st} with $t=\theta$ and $l=1$, \eqref{eq:maj_der_phi_n}, and $\abs{\varphi_n} \leqslant 1$ to get
\begin{align*}
&\abs{\frac{\partial}{\partial t} \left[e^{t^2/2}\varphi_n^{N_n}\left(\frac{s}{\bnxm},\frac{t}{\bnym}\right)\right]_{t=\theta}}\\
& = e^{\theta^2/2}\left|\varphi_n^{N_n-1}\left(\frac{s}{\bnxm},\frac{\theta}{\bnym}\right)\right| \cdot\left|\theta \varphi_n\left(\frac{s}{\bnxm},\frac{\theta}{\bnym}\right) \right. \\
& \hspace{7cm} \left. + \frac{N_n^{1/2}}{\ssdd{\rva{Y}{n}}} \frac{\partial \varphi_n}{\partial t} \left(\frac{s}{\bnxm},\frac{\theta}{\bnym}\right)\right|\\
&\leqp  (\abs{s} + 2\abs{\theta})e^{\theta^2/2-(s^2+\theta^2) \cdot \cfcX(N_n-1)/N_n},
\end{align*}
for $s\in A_2$ and $u$ and $\theta$ as in the integral in the statement of the Lemma. Finally, using \eqref{def:I_2}, we get that, for $N_n \geqslant 2$,
\begin{align*}
 & \int_0^{\eta \bnym} \sup_{0 \leqp \theta \leqp u} I_2(n, u, \theta) du \\
 & \leqp 2 \cn^{-1} \int_0^{+\infty} \int_{\varepsilon \ssdd{\rva{X}{n}}N_n^{1/2}}^{+\infty} \sup_{0 \leqslant \theta \leqslant u} \bigg[ (s + 2\theta) \exp \bigg( \frac{\theta^2}{2} \bigg( 1 - 2\cfcX\frac{N_n-1}{N_n} \bigg) \bigg) \bigg] \\
 & \hspace{10cm}\cdot e^{-u^2/2-s^2 \cdot \cfcX(N_n-1)/N_n} ds du \\
 & \leqslant 2 \ctll^{-1} \int_0^{+\infty} \int_{\varepsilon \ssdd{\rva{X}{n}}N_n^{1/2}}^{+\infty} (s + 2u) e^{-\min(1, \cfcX)u^2/2-s^2 \cfcX/2} ds du \\
 & \leqslant   e^{-N_n \cfcX \varepsilon^2 \ssdd{\rva{X}{n}}^2/2} \left(\frac{\ctll^{-1} \cfcX^{-1}\sqrt{2\pi}}{\sqrt{\min(1, \cfcX)}} + \frac{ 4 \ctll^{-1}}{\min(1, \cfcX)} \frac{1}{\cfcX\varepsilon \ssdd{\rva{X}{n}} N_n^{1/2}}\right)\\
 &\leqslant C_2'e^{-C_2''N_n}
\end{align*}
where 
\begin{equation*} 
C_2' \defeq  \ctll^{-1}\cfcX^{-1} \left( \frac{\sqrt{2\pi}}{\sqrt{\min(1, \cfcX)}} + \frac{4}{\min(1, \cfcX) \min \bigg( \frac{2}{9} (\cvarX \crhoX)^{-1}, \pi \bigg) \cvarXtilde} \right)
\end{equation*}
and $C_2'' \defeq \cvarXtilde^2/2 \cfcX \min \bigg( \frac{2}{9} (\cvarX \crhoX)^{-1}, \pi \bigg)^2 $.
The result follows, writing
\[
C_2' e^{-C_2'' N_n} = \frac{C_2' (C_2'')^{-1/2}}{N_n^{1/2}} (C_3 N_n)^{1/2} e^{- C_3 N_n} \leqslant \frac{C_2' (C_2'')^{-1/2}}{N_n^{1/2}} (1/2)^{1/2} e^{-1/2} \eqdef \frac{C_2}{N_n^{1/2}},
\]
since $x^{1/2} e^{-x}$ is maximum in $1/2$.
\end{proof}

\bibliographystyle{plain}
\bibliography{biblio_gde_dev}

\begin{thebibliography}{10}

\bibitem{Chassaing02}
P.~Chassaing and G.~Louchard.
\newblock Phase transition for parking blocks, {B}rownian excursion and
  coalescence.
\newblock {\em Random Structures Algorithms}, 21(1):76--119, 2002.

\bibitem{ChF03}
Philippe Chassaing and Philippe Flajolet.
\newblock Hachage, arbres, chemins \& graphes.
\newblock {\em Gaz. Math.}, (95):29--49, 2003.

\bibitem{Feller68}
W.~Feller.
\newblock {\em An introduction to probability theory and its applications.
  {V}ol. {I}}.
\newblock Third edition. John Wiley \& Sons Inc., New York, 1968.

\bibitem{Feller71}
W.~Feller.
\newblock {\em An introduction to probability theory and its applications.
  {V}ol. {II}.}
\newblock Second edition. John Wiley \& Sons Inc., New York, 1971.

\bibitem{FPV98}
P.~Flajolet, P.~Poblete, and A.~Viola.
\newblock On the analysis of linear probing hashing.
\newblock {\em Algorithmica}, 22(4):490--515, 1998.
\newblock Average-case analysis of algorithms.

\bibitem{Holst79}
L.~Holst.
\newblock Two conditional limit theorems with applications.
\newblock {\em Ann. Statist.}, 7(3):551--557, 1979.

\bibitem{Janson01a}
S.~Janson.
\newblock Asymptotic distribution for the cost of linear probing hashing.
\newblock {\em Random Structures Algorithms}, 19(3-4):438--471, 2001.
\newblock Analysis of algorithms (Krynica Morska, 2000).

\bibitem{Janson01}
S.~Janson.
\newblock Moment convergence in conditional limit theorems.
\newblock {\em J. Appl. Probab.}, 38(2):421--437, 2001.

\bibitem{Janson05}
S.~Janson.
\newblock Individual displacements for linear probing hashing with different
  insertion policies.
\newblock {\em ACM Trans. Algorithms}, 1(2):177--213, 2005.

\bibitem{Knuth74}
D.~E. Knuth.
\newblock Computer science and its relation to mathematics.
\newblock {\em Amer. Math. Monthly}, 81:323--343, 1974.

\bibitem{Knuth98a}
D.~E. Knuth.
\newblock Linear probing and graphs.
\newblock {\em Algorithmica}, 22(4):561--568, 1998.
\newblock Average-case analysis of algorithms.

\bibitem{Kolchin84}
V.~F. Kolchin.
\newblock {\em Random mappings}.
\newblock Translation Series in Mathematics and Engineering. Optimization
  Software, Inc., Publications Division, New York, 1986.
\newblock Translated from the Russian, With a foreword by S. R. S. Varadhan.

\bibitem{Loeve55}
M.~Lo{\`e}ve.
\newblock {\em Probability theory. {F}oundations. {R}andom sequences}.
\newblock D. Van Nostrand Company, Inc., Toronto-New York-London, 1955.

\bibitem{Marckert01-1}
J.-F. Marckert.
\newblock Parking with density.
\newblock {\em Random Structures Algorithms}, 18(4):364--380, 2001.

\bibitem{Pavlov77}
Y.~L. Pavlov.
\newblock Limit theorems for the number of trees of a given size in a random
  forest.
\newblock {\em Mat. Sb. (N.S.)}, 103(145)(3):392--403, 464, 1977.

\bibitem{Pavlov96}
Y.~L. Pavlov.
\newblock Random forests.
\newblock In {\em Probabilistic methods in discrete mathematics
  ({P}etrozavodsk, 1996)}, pages 11--18. VSP, Utrecht, 1997.

\bibitem{QR82}
M.~P. Quine and J.~Robinson.
\newblock A {B}erry-{E}sseen bound for an occupancy problem.
\newblock {\em Ann. Probab.}, 10(3):663--671, 1982.

\bibitem{Wendel75}
J.~G. Wendel.
\newblock Left-continuous random walk and the {L}agrange expansion.
\newblock {\em Amer. Math. Monthly}, 82:494--499, 1975.

\end{thebibliography}

\end{document}